\documentclass[12pt]{amsart}  

\usepackage{custom_preamble} 

\begin{document}

\title[Approximate homomorphisms]{Approximate homomorphisms between the {B}oolean cube and groups of prime order}

\author{\tsname}
\address{\tsaddress}
\email{\tsemail}

\maketitle

\setcounter{section}{1}

The purpose of this note is to highlight a question raised by Shachar Lovett \cite{lov::0}, and to offer some motivation for its study.  Our interest is in the existence of injections $f:(\Z/2\Z)^n\rightarrow \Z/p\Z$ (where $p$ is an odd prime) that are `approximately homomorphisms' in the sense that $f(x+y)=f(x)+f(y)$ for many $x$ and $y$.  Lovett noted that if $p > 2^n$ then the map $f:(\Z/2\Z)^n \rightarrow \Z/p\Z$ that `embeds in binary' is an injection with\footnote{The map is defined by $f:(\Z/2\Z)^n \rightarrow \Z/p\Z; (x_1,\dots,x_n) \mapsto x_1+2x_2+\dots+2^{n-1}x_n$, and the probability that $f(x+y)=f(x)+f(y)$ is the probability that we do not need to make a carry when adding $f(x)$ and $f(y)$.  Indeed, $f(x)+f(y) -f(x+y)= 2x_1y_1+4x_2y_2 + \dots + 2^{n}x_ny_n$.  Since $2$ is invertible in $\Z/p\Z$ we can divide and then the right hand side is between $0$ and $2^{n}-1$.  It follows that it equals $0$ if and only if $x_iy_i=0$ for all $1 \leq i \leq n$ from which the equality follows.}
\begin{equation*}
\P(f(x+y)=f(x)+f(y))=\left(\frac{3}{4}\right)^n,
\end{equation*}
where $x$ and $y$ are taken uniformly and independently from $(\Z/2\Z)^n$.  We shall show the following complementary result.
\begin{proposition}\label{prop.1}
Suppose that $f:(\Z/2\Z)^n \rightarrow \Z/p\Z$ is an injection.
Then
\begin{equation*}
\mathbb{P}(f(x+y)=f(x)+f(y))=O\left(2^{-\frac{n}{11}}\right).
\end{equation*}
\end{proposition}
The method we use does generalise to other Abelian groups (see Theorem \ref{thm.d} below), but the next proposition captures perhaps the most interesting consequence.  If $G$ is a finite Abelian group then the structure theorem tells us there is a unique sequence of natural numbers $1<d_1\divides d_2 \divides \cdots  \divides d_n$ such that $G \cong (\Z/d_1\Z)\oplus \cdots \oplus (\Z/d_n\Z)$, and we say that $G$ has $n$ invariant factors.
\begin{proposition}\label{prop.2}
Suppose that $G$ and $H$ are finite Abelian groups of co-prime order; $G$ has $n$ invariant factors; and $f:G \rightarrow H$ is an injection.  Then
\begin{equation*}
\mathbb{P}(f(x+y)=f(x)+f(y))=O\left(17^{-\frac{n}{79}}\right).
\end{equation*}
\end{proposition}
For comparison, $2^{-\frac{1}{11}} = 0.939\cdots < 17^{-\frac{1}{79}}=0.964\cdots < 2^{-\frac{1}{20}}$.

We have presented Proposition \ref{prop.2} because it makes essential use of a result of Bukh \cite[Theorem 3]{buk::0} (formulated for quantitative reasons in the relevant special case in Lemma \ref{lem.buk} below), without which the constant $17^{-\frac{1}{79}}$ would not be absolute\footnote{As a concrete example, without Bukh's result our arguments would give no information when $G=(\Z/p_n\Z)^n$ where $p_n$ is the $n$th prime.}.  Indeed, Bukh's work raises a number of questions \emph{e.g.} \cite[Question 16]{buk::0} around sumset inequalities, that bear on the problems of this paper.

We now state our main result for which we require a little more notation.  Given an Abelian group $G$ and $r \in \Z$ we write $K_{G,r}$ for the kernel of the homomorphism $G \rightarrow G; x \mapsto rx$, and given $A \subset G$ write $r\cdot A:=\{ra: a \in A\}$.

Our main result is the following.
\begin{theorem}\label{thm.d}
Suppose that $G$ and $H$ are Abelian groups, $G$ is finite, $r \in \N$ is a parameter, and $f:G \rightarrow H$ is an injection.  Then
\begin{equation*}
\P(f(x+y)=f(x)+f(y)) = O\left(\frac{1}{|G|}\min\left\{|r\cdot G||K_{H,r}|, |r\cdot H||K_{G,r}|\right\}\right)^\alpha
\end{equation*}
for any $\alpha\leq \max\left\{\frac{1}{5r+1},\frac{1}{18\lfloor \log_2 r\rfloor +7}\right\}$.
\end{theorem}
Although the statement may appear rather gruesome, Proposition \ref{prop.1} follows immediately on taking $r=2$ since $2\cdot (\Z/2\Z)^n=\{0_{(\Z/2\Z)^n}\}$ and $K_{\Z/p\Z,2}=\{0_{\Z/p\Z}\}$ for $p$ odd.
\begin{proof}[Proof of Proposition \ref{prop.2}]
Suppose that $d_1\divides \cdots \divides d_n$ are the invariant factors of $G$.  Let $r \divides d_1$ be a prime.  Since $G$ and $H$ have co-prime order we see that $K_{H,r}$ is trivial and $|r\cdot G| = \frac{d_1}{r}\cdots \frac{d_n}{r}$.  Applying Theorem \ref{thm.d} we see that
\begin{equation*}
\mathbb{P}(f(x+y)=f(x)+f(y))=O\left(2^{-c(r)n}\right) \text{ where }c(r)= \max\left\{\frac{\log_2 r}{5r+1},\frac{\log_2 r}{18\lfloor \log_2 r\rfloor +7}\right\}.
\end{equation*}
If $p \geq 2^7$ then
\begin{equation*}
c(p) \geq \frac{\log_2 p}{18 \log_2 p +7} =\frac{1}{18+\frac{7}{\log_2p}} \geq \frac{1}{19}>\frac{\log_217}{79}= c(17).
\end{equation*}
It follows that $c(r) \geq \min\{c(p): p \leq 2^7 \text{ is prime}\}$.  A short calculation shows that this minimum is $c(17)$ and the result is proved.
\end{proof}
Theorem \ref{thm.d} can be applied (again with $r=2$) to show that $\P(f(x+y)=f(x)+f(y))=O(p^{-\frac{1}{11}})$ for injections $f:\Z/p\Z \rightarrow (\Z/2\Z)^n$ (an example where Proposition \ref{prop.2} gives nothing); and $\P(f(x+y)=f(x)+f(y)) = O\left(2^{-\frac{n}{11}}\right)$ for injections $f:(\Z/2\Z)^{2n}\rightarrow (\Z/4\Z)^n$ (an example where Proposition \ref{prop.2} does not apply).

The centred unwrapping map from $f:\Z/p\Z \rightarrow \Z$ taking $x+p\Z$ to $x$ whenever $x \in \left(-\frac{p}{2},\frac{p}{2}\right]$ is injective and (supposing that $p=2k+1$ is odd) has
\begin{align*}
\P(f(x+y)=f(x)+f(y)) & \geq \frac{1}{p^2}\sum_{-\frac{p}{2}< x,y \leq \frac{p}{2}}{1_{\left(-\frac{p}{2},\frac{p}{2}\right]}(x+y)}\\
& =\frac{1}{p^2}\sum_{x=-k}^k{\sum_{y=\max\{-k,-k-x\}}^{\min\{k,k-x\}}{1}} = \frac{3k^2+3k+1}{4k^2+4k+1} > \frac{3}{4}.
\end{align*}
If $q>p$ is another prime then this can be composed with the natural projection $\Z\rightarrow \Z/q\Z$ to give an injection $f:\Z/p\Z\rightarrow \Z/q\Z$ such that $\P(f(x+y)=f(x)+f(y)) > \frac{3}{4}$.

As a final remark we mention `approximate homomorphisms' between possibly non-Abelian groups have been studied by Moore and Russell.  As an example it follows from \cite[Theorem 3]{moorus::0} and a classical result of Frobenius\footnote{Specifically \cite[Theorem 3.5.1]{davsarval::} that for $q \geq 5$, $\PSL_2(q)$ has no non-trivial representation of dimension below $\frac{1}{2}(q-1)$.} that $\P(f(xy)=f(x)+f(y))=O(|\PSL_2(q)|^{-\frac{1}{6}})$ for injections $f:\PSL_2(q) \rightarrow H$ where $q \geq 5$ and $H$ is Abelian.

In our arguments we regard our groups as endowed with counting measure, so that if $G$ is an Abelian group we define
\begin{equation*}
f\ast g(x):=\sum_y{f(y)g(x-y)} \text{ for all }x \in G \text{ and }f,g \in \ell_1(G);
\end{equation*}
and similarly
\begin{equation*}
\langle f,g\rangle_{\ell_2(G)}:=\sum_y{f(y)\overline{g(y)}} \text{ for all }f,g\in \ell_2(G).
\end{equation*}
In particular, if $G$ and $H$ are Abelian groups with $G$ finite, $f:G \rightarrow H$ is a function, and $\Gamma:=\{(x,f(x)):x \in G\}$ is the graph of $f$, then $|\Gamma|=|G|$ and
\begin{align}\label{eqn.j}
\frac{1}{|\Gamma|^2}\langle 1_\Gamma \ast 1_\Gamma,1_\Gamma\rangle_{\ell_2(G\times H)} &= \frac{1}{|G|^2}|\{a,b \in G: f(a+b)=f(a)+f(b)\}|\\\nonumber  &=\P(f(x+y)=f(x)+f(y)).
\end{align}

We now turn to the proof of Theorem \ref{thm.d} which is primarily done through Proposition \ref{prop.k}.  The overall structure of the argument is a common one in additive combinatorics.  We start with the closed graph theorem for groups: that is, the observation that a function between groups is a homomorphism if and only if its graph is a subgroup of the direct product.

The probability that we are interested in (\ref{eqn.j}) measures how close the graph of the function is to being a group.  There is a well-developed theory, starting with work of Balog and Szemer{\'e}di (see \cite[\S2.5]{taovu::}), describing what such sets must look like and it turns out that in Abelian groups of bounded exponent they must be close to genuine subgroups.  This is not quite the situation we are in, but arguments of this type can be used to show that if the probability in (\ref{eqn.j}) is large then the function must agree with a genuine homomorphism on a large set.  This leads to a contradiction.  

For our arguments we do not need much of this general theory as a counting argument lets us arrive at a contradiction directly.  This is essentially the content of Claim \ref{cl.s} below.
\begin{proposition}\label{prop.k}
Suppose that $G$ and $H$ are Abelian groups; $G$ is finite; $\Gamma \subset G \times H$ is such that the coordinate projections restricted to $\Gamma$ are injective; and $r \in \N$ is a parameter.  Then
\begin{equation}\label{eqn.bg}
\frac{1}{|\Gamma|^2}\langle 1_\Gamma \ast 1_\Gamma,1_\Gamma\rangle_{\ell_2(G\times H)} = O\left(\frac{|r\cdot G||K_{H,r}|}{|\Gamma|}\right)^{\alpha}
\end{equation}
for any $\alpha\leq \max\left\{\frac{1}{5r+1},\frac{1}{18\lfloor \log_2 r\rfloor +7}\right\}$.
\end{proposition}
\begin{proof}
Write $\pi_G$ and $\pi_H$ for the coordinate projections on $G \times H$.
\begin{claims}\label{cl.s}
Suppose that $X,B \subset \Gamma$.  Then
\begin{equation*}
\frac{|X+r\cdot B|}{|X|} \geq \frac{|B|}{|K_{H,r}||r\cdot G|}.
\end{equation*}
\end{claims}
\begin{proof}
Write $\mathcal{Q}:=\{(x,y,z,w)\in X\times (r\cdot B) \times X \times (r\cdot B) : x+y=z+w\}$, and consider the map
\begin{equation*}
\psi: \mathcal{Q} \rightarrow X \times (r\cdot B) \times (r\cdot G); (x,y,z,w) \mapsto (x,y,\pi_G(z)-\pi_G(x)).
\end{equation*}
This is well-defined: if $(x,y,z,w) \in \mathcal{Q}$ then $x+y=z+w$, and $\pi_G$ is a homomorphism so
\begin{equation*}
\pi_G(z)-\pi_G(x) = \pi_G(y)-\pi_G(w) \in \pi_G(r\cdot B) - \pi_G(r\cdot B) \subset r\cdot G.
\end{equation*}
Moreover, $\psi$ is an injection: suppose that $\psi(x,y,z,w)=\psi(x',y',z',w')$.  Then $x=x'$, $y=y'$, and $\pi_G(z)-\pi_G(x)=\pi_G(z')-\pi_G(x')$.  It follows that $\pi_G(z)=\pi_G(z')$ and hence $z=z'$ since $\pi_G$ restricted to $\Gamma$ is injective.  But then $w'=x'+y'-z' = x+y-z = w$ and the injectivity of $\psi$ follows.

We conclude
\begin{equation}\label{eqn.y}
\|1_{X} \ast 1_{r\cdot B}\|_{\ell_2(G\times H)}^2 = |\mathcal{Q}|  \leq |X||r\cdot B||r\cdot G|.
\end{equation}
On the other hand, if $|X+r\cdot B|\leq L|X|$ then by the Cauchy-Schwarz inequality we have
\begin{equation*}
\|1_{X} \ast 1_{r\cdot B}\|_{\ell_2(G\times H)}^2 \geq \frac{\|1_{X} \ast 1_{r\cdot B}\|_{\ell_1(G\times H)}^2}{|X+r\cdot B|} = \frac{|X|^2|r\cdot B|^2}{|X+r\cdot B|}\geq \frac{1}{L}|X||r\cdot B|^2.
\end{equation*}
Combining with (\ref{eqn.y}) and cancelling we see that
\begin{align*}
|r\cdot G| \geq \frac{1}{L}|r\cdot B|  \geq \frac{1}{L}|\pi_H(r\cdot B)| = \frac{1}{L}|r\cdot \pi_H( B)| \geq \frac{1}{L}\frac{|\pi_H(B)|}{|K_{H,r}|}=\frac{1}{L}\frac{|B|}{|K_{H,r}|},
\end{align*}
since $\pi_H$ restricted to $\Gamma$ is an injection. The claim follows.
\end{proof}
Write $\epsilon$ for the left hand side of (\ref{eqn.bg}). 
\begin{claims}
(\ref{eqn.bg}) holds with $\alpha \leq \frac{1}{5r+1}$.
\end{claims}
\begin{proof}
Define a bipartite graph $\mathcal{G}$ with vertex sets two copies of $\Gamma$, and $(x,y) \in E(\mathcal{G})$ if and only if $x+y \in \Gamma$.  Then $|E(\mathcal{G})| = \epsilon |\Gamma |^2$, and $\Gamma +_{\mathcal{G}} \Gamma:=\{x+y: (x,y) \in E(\mathcal{G})\} \subset \Gamma$, so $|\Gamma +_{\mathcal{G}} \Gamma| \leq |\Gamma|$. The Balog-Szemer{\'e}di-Gowers Lemma \cite[Theorem 2.29]{taovu::} can then be applied to give sets $A,B \subset \Gamma$ such that
\begin{equation*}
|A|,|B| = \Omega\left( \epsilon|\Gamma|\right) \text{ and } |A+B| = O(\epsilon^{-4}|\Gamma|) = O(\epsilon^{-5}|A|).
\end{equation*}
Apply Pl{\"u}nnecke's inequality \cite[Corollary 6.26]{taovu::} to get a non-empty $X \subset A$ such that (recalling the notation $rB:=B+\cdots + B$) $|X+rB| = O\left( \epsilon^{-5r}|X|\right)$.  But then $r\cdot B \subset rB$ and so $|X+r\cdot B| \leq |X+rB|=O\left( \epsilon^{-5r}|X|\right)$ and we get the claimed bound from Claim A and the fact that $|B|=\Omega(\epsilon |\Gamma|)$.
\end{proof}
\begin{claims}
(\ref{eqn.bg}) holds with $\alpha \leq \frac{1}{18\lfloor \log_2 r\rfloor +7}$.
\end{claims}
\begin{proof}
Apply Lemma \ref{lem.bsg} to get a set $B \subset \Gamma$ such that $x_0-B=B$ for some $x_0 \in G$ and
\begin{equation*}
|B| = \Omega\left( \epsilon|\Gamma|\right) \text{ and } |B-B| = O(\epsilon^{-5}|\Gamma|) = O(\epsilon^{-6}|B|).
\end{equation*}
Apply Pl{\"u}nnecke's inequality \cite[Corollary 6.26]{taovu::} to get a non-empty $Y \subset B$ such that $|Y-3B| = O\left( \epsilon^{-18}|Y|\right)$.  Since $2\cdot B  \subset 2B$ we get $|Y-B-2\cdot B| \leq O\left( \epsilon^{-18}|Y|\right)$.  By Lemma \ref{lem.buk} (and the fact that $|B+B|=|B+x_0-B| = |B-B|$)
\begin{equation*}
|B+r\cdot B| = O\left(\epsilon^{-18}\right)^{\lfloor\log_2 r\rfloor}|B+B|=O\left(\epsilon^{-18}\right)^{\lfloor\log_2 r\rfloor}O(\epsilon^{-6}|B|).
\end{equation*}
We get the claimed bound from Claim \ref{cl.s} with $X=B$, and the fact that $|B|=\Omega(\epsilon |\Gamma|)$.
\end{proof}
\end{proof}
\begin{proof}[Proof of Theorem \ref{thm.d}]
We can apply Proposition \ref{prop.k} to the graph of $f$ \emph{i.e.} $\Gamma=\{(x,f(x)): x \in G\}$ in $G\times H$.  ($\pi_G$ is injective on $\Gamma$ since $f$ is a function, and $\pi_H$ is injective on $\Gamma$ since $f$ is injective.). The calculation in (\ref{eqn.j}) then gives the bound in the first term in the minimum in Theorem \ref{thm.d}.  A similar argument, applying Proposition \ref{prop.k} with $\Gamma=\{(f(x),x):x \in G\}$ and $G$ and $H$ swapped gives the bound in the second term in the minimum Theorem \ref{thm.d}.
\end{proof}

\section{Sumset estimates}

We have taken some care with powers in this note which means that we have needed bespoke versions of Bukh's Theorem \cite[Theorem 3]{buk::0} and the Balog-Szemer{\'e}di-Gowers Lemma \cite[Theorem 2.29]{taovu::}.  These are proved below through minor variations on the arguments referenced.

The first lemma is proved in roughly the same way as \cite[Theorem 15]{buk::0}, though we have to take care because our group may have torsion.  In particular, even if $|A+B| \leq K|A|$ we need not have $|2\cdot A + 2\cdot B| \leq K|2\cdot A|$.  This can be seen through a minor variation of \cite[Exercise 6.5.10]{taovu::}: let $d\in \N_0$ and $G:=(\Z/4\Z)^d\times \Z$, and define
\begin{equation*}
A:=(\Z/4\Z)^d\times \{0\} \cup \{0+4\Z\}^d \times \{1,\dots,2^d\} \text{ and }B:=\{0+4\Z,1+4\Z\}^d \times \{0\}.
\end{equation*}
Then we can compute
\begin{equation*}
|A| = 4^d +2^d, |2\cdot A| = 2\cdot 2^{d}, |A+B| = 2\cdot 4^d, \text{ and } |2\cdot A + 2\cdot B| = 2^{2d}+2^d,
\end{equation*}
from which it follows that
\begin{equation*}
\frac{|A+B|}{|A|} \leq 2 \text{ and } \frac{|2\cdot A  + 2\cdot B|}{|2\cdot A|} \geq 2^{d-1}.
\end{equation*}
To navigate around this we appeal to the beautiful \cite[Proposition 2.1]{pet::} of Petridis, the central component in his (2nd) proof of Pl{\"u}nnecke's inequality.
\begin{lemma}\label{lem.buk}
Suppose that $G$ is an Abelian group; $Y\subset A \subset G$ is finite with $|Y -A- 2\cdot A| \leq K|Y|$; and $r \in \N$ is a parameter.  Then
\begin{equation*}
|X+r\cdot A| \leq K^{\lfloor \log_2 r\rfloor}|X+A|\text{ for all  finite }X \subset G.
\end{equation*}
\end{lemma}
\begin{proof}
Let $\emptyset \neq Z \subset Y$ be such that $\frac{|Z-A-2\cdot A|}{|Z|}$ is minimal.  By Petridis' lemma \cite[Proposition 2.1]{pet::} we see that
\begin{equation}\label{eqn.pet}
|Z-A-2\cdot A+C| \leq \frac{|Z-A-2\cdot A|}{|Z|}|Z+C| \leq K|Z+C| \text{ for all finite }C \subset G.
\end{equation}
Suppose $k \in \N$.  We may assume that $G$ is finitely generated (\emph{e.g.} by $A \cup X$), so $K_{G,2^{k-1}}$ is finite.  Then by (\ref{eqn.pet}) we have
\begin{equation}\label{eqn.l}
|2^{k-1}\cdot (Z-A-2\cdot A)| = \frac{|Z-A-2\cdot A + K_{G,2^{k-1}}|}{|K_{G,2^{k-1}}|} \leq \frac{K|Z+K_{G,2^{k-1}}|}{|K_{G,2^{k-1}}|} = K|2^{k-1}\cdot Z|.
\end{equation}
Now, apply the Ruzsa triangle inequality \cite[Lemma 2.6]{taovu::} to see that
\begin{align*}
\left|X+\sum_{i=0}^k{2^i\cdot A}\right| & \leq \frac{\left|\left(X+\sum_{i=0}^{k-2}{2^i\cdot A}\right) +2^{k-1}\cdot Z\right|\left|-2^{k-1}\cdot Z -\left(-2^{k-1}\cdot A - 2^k\cdot A\right)\right|}{|-2^{k-1}\cdot Z|}\\
& \leq \left|X+\sum_{i=0}^{k-1}{2^i\cdot A}\right| \cdot \frac{\left|2^{k-1}\cdot (Z-A-2\cdot A)\right|}{|2^{k-1}\cdot Z|} \leq K\left|X+\sum_{i=0}^{k-1}{2^i\cdot A}\right|,
\end{align*}
by (\ref{eqn.l}) and the fact that $Z \subset Y \subset A$.  It follows by induction that $\left|X+\sum_{i=0}^k{2^i\cdot A}\right| \leq K^{k}|X+A|$ for all $k \in \N_0$. 

Let $k:=\lfloor \log_2r\rfloor$ and write $r$ in binary \emph{i.e.} let $\epsilon_0,\dots,\epsilon_k \in \{0,1\}$ be such that $r=\epsilon_0+\epsilon_12+\cdots + \epsilon_k2^k$.  Then
\begin{equation*}
|X+r\cdot A| \leq \left|X+\sum_{i=0}^k{\epsilon_i 2^i\cdot A}\right| \leq \left|X+\sum_{i=0}^k{2^i\cdot A}\right| \leq K^{k}|X+A|
\end{equation*}
as claimed.
\end{proof}
The second lemma varies the vanilla Balog-Szemer{\'e}di-Gowers Lemma by adding some $x_0$ with $x_0-T=T$.  This means that for sumset purposes we can treat $T$ as symmetric and will not need to bear the cost of applying P{\"u}nnecke's inequality to pass between bounds on $|T-T|$ and $|T+T|$.
\begin{lemma}\label{lem.bsg}
Suppose that $G$ is an Abelian group and $S \subset G$ has $\langle 1_S \ast 1_S , 1_S \rangle_{\ell_2(G)} \geq \epsilon |S|^2$.  Then there is a set $T \subset S$ and some $x_0 \in G$ such that
\begin{equation*}
x_0-T=T, |T| = \Omega(\epsilon |S|) \text{ and } |T-T| =O(\epsilon^{-5}|S|).
\end{equation*}
\end{lemma}
\begin{proof}
For $y,z \in G$ define
\begin{equation*}
p(y,z):=\frac{1}{|S|}|(y+S)\cap S\cap (S+z)|1_S(y)1_S(z) \leq \frac{1}{|S|}1_S \ast 1_{-S}(y-z).
\end{equation*}
Let $x \in S$ be chosen uniformly at random and put $U:=S \cap (x-S)$ so that $\P((y,z) \in U^2) =p(y,z)$. Let $c:=\frac{1}{18}$ (for reasons that will become clear) and note, by the Cauchy-Schwarz inequality, that
\begin{equation*}
\E{\left(|U|^2 - \frac{1}{2}c^{-1}|\{(y,z)\in U^2: p(y,z) \leq c\epsilon^2 \}|\right)} \geq \E{|U|^2} - \frac{1}{2}\epsilon^2 |S|^2 \geq \frac{1}{2}\epsilon^2|S|^2.
\end{equation*}
It follows that we can pick $x \in S$ such that
\begin{equation*}
|U| \geq \frac{\epsilon}{\sqrt{2}}|S| \text{ and }|\{(y,z)\in U^2: p(y,z) > c\epsilon^2 \}| \geq (1-2c)|U|^2.
\end{equation*}
By averaging, the set $R:=\left\{y \in U: |\{z \in U: p(y,z) > c\epsilon^2 \}|  \geq (1-6c)|U|\right\}$ has $|R|\geq \frac{2}{3}|U|$.  Since $x-U=U$ we see that $x-R \subset U$ and hence $T:=R\cap (x-R)$ has $|T| \geq \frac{1}{3}|U|$ (by the pigeon-hole principle) and $T=x-T$.

For each $s \in T-T$ there are elements $y,w \in T$ such that $s=y-w$.  Since $y,w \in R$ and $1-6c \geq \frac{2}{3}$, the pigeon-hole principle tells us that there are at least $\frac{1}{3}|U|$ elements $z \in U$ such that $p(y,z)>c\epsilon^2$ and $p(w,z)>c\epsilon^2$.  Hence
\begin{align*}
\frac{|U|}{3}\cdot (c\epsilon^2)^2|S|^2 & < |S|^2\sum_{z \in U}{p(y,z)p(w,z)}\\ & \leq \sum_{z\in U}{1_S\ast1_{-S}(y-z)1_{-S}\ast 1_{S}(z-w)} \leq 1_S \ast 1_{-S} \ast 1_{-S}\ast 1_{S}(y-w).
\end{align*}
It follows that
\begin{equation*}
|T-T| \cdot \frac{|U|}{3}\cdot (c\epsilon^2)^2|S|^2 \leq \sum_{s \in G}{1_S \ast 1_{-S} \ast 1_S\ast 1_{-S}(s)} = |S|^4,
\end{equation*}
and the lemma is proved with $x_0=x$.
\end{proof}

\section*{Acknowledgements}

The author thanks Terry Tao for bringing the problem to his attention, an anonymous referee for their improvements, and Ruoyi Wang for a correction.

\bibliographystyle{halpha}

\bibliography{references}

\end{document}